\documentclass{amsart}

\parskip = \bigskipamount
\usepackage[margin=1.5in]{geometry}
\usepackage{amssymb}
\usepackage{amscd}
\usepackage[all]{xy}
\usepackage{bbm}
\usepackage{mathrsfs}
\usepackage{enumerate}
\usepackage{tikz}
\usepackage{geometry}
\usepackage{stmaryrd}
\usepackage{etoolbox}
\usepackage{array}
\usepackage{hyperref}

\newcommand{\N}{\mathbb{N}}
\newcommand{\Z}{\mathbb{Z}}

\newcommand{\R}{\mathbb{R}}

\newcommand{\inv}{^{-1}}

\newcommand{\eps}{\varepsilon}

\newcommand{\del}{\nabla}
\newcommand{\ind}{\operatorname{ind}}

\newcommand{\lap}{\Delta}

\newcommand{\bd}{\partial}
\newcommand{\cl}{\overline}

\newcommand{\weak}{\rightharpoonup}

\newcommand{\RP}{{\R \text{\normalfont P}}}

\newcommand{\grad}{\del}
\newcommand{\vol}{\operatorname{Vol}}

\newcommand{\f}{\colon}
\newcommand{\area}{\operatorname{Area}}

\newcommand{\h}{\mathcal H} 
\newcommand{\hv}{\mathfrak h}

\newcommand{\z}{\mathcal Z}
\newcommand{\B}{\mathcal B}
\newcommand{\C}{\mathcal C}
\newcommand{\sd}{\operatorname{\Delta}}

\theoremstyle{plain}
\newtheorem{theorem}{Theorem}
\newtheorem{corollary}[theorem]{Corollary}
\newtheorem{prop}[theorem]{Proposition}
\newtheorem{lem}[theorem]{Lemma}

\newtheorem{theo}{Theorem}

\theoremstyle{definition}
\newtheorem{defn}[theorem]{Definition}
\newtheorem{rem}[theorem]{Remark}

\begin{document}

\title{The Half-Volume Spectrum of a Manifold}
\author{Liam Mazurowski}
\author{Xin Zhou}
\address{Cornell University, Department of Mathematics, Ithaca, New York 14850}
\email{lmm334@cornell.edu}
\address{Cornell University, Department of Mathematics, Ithaca, New York 14850}
\email{xinzhou@cornell.edu}

\begin{abstract}
We define the half-volume spectrum $\{{\tilde \omega_p\}_{p\in \N}}$ of a closed manifold $(M^{n+1},g)$.  This is analogous to the usual volume spectrum of $M$, except that we restrict to $p$-sweepouts whose slices each enclose half the volume of $M$. We prove that the Weyl law continues to hold for the half-volume spectrum.  We define an analogous half-volume spectrum $\tilde c(p)$ in the phase transition setting. Moreover, for $3 \le n+1 \le 7$, we use the Allen-Cahn min-max theory to show that each $\tilde c(p)$ is achieved by a constant mean curvature surface enclosing half the volume of $M$ plus a (possibly empty) collection of minimal surfaces with even multiplicities. 
\end{abstract}

\maketitle

\section{Introduction}

The spectrum of the Laplacian is an important invariant of a closed Riemannian manifold $(M^{n+1},g)$.  A number $\lambda$ is called an eigenvalue of the Laplacian provided there is a function $u\f M\to \R$ such that $\lap u + \lambda u = 0$.  It is well-known that the eigenvalues form a discrete sequence $0 = \lambda_0 < \lambda_1 \le \lambda_2 \le \hdots$ and $\lambda_p \to \infty$ as $p\to \infty$.  In fact, the eigenvalues of Laplacian are characterized by the min-max formula
\[
\lambda_p = \inf_{(p+1)\text{-planes } P \subset W^{1,2}(M)} \left[\sup_{u\in P\setminus\{0\}} \frac{\int_M \vert \grad u\vert^2}{\int_M u^2}\right],
\]
and they satisfy the Weyl law 
\[
\lambda_p \sim {4\pi^2}\vol(B)^{-\frac{2}{n+1}}\vol(M)^{-\frac{2}{n+1}}p^{\frac{2}{n+1}}
\]
as $p\to \infty$. 
Here $B$ is the unit ball in $\R^{n+1}$. 

In \cite{G}, Gromov proposed a non-linear analog of the spectrum of the Laplacian.  Roughly speaking, he defines a $p$-sweepout of $M$ to be a family $X$ of hypersurfaces  with the following property: given any $p$ points in $M$, there is a hypersurface $\Sigma$ belonging to the family $X$ which passes through all $p$ of these points. Then he defines the $p$-widths 
\[
\omega_p = \inf_{p\text{-sweepouts } X} \left[\sup_{\Sigma\in X} \area(\Sigma)\right].
\]
See Section \ref{ap} for precise definitions. The sequence $\{\omega_p\}_{p\in \N}$ is called the volume spectrum of $M$. 

Gromov \cite{G1} and Guth \cite{Gu} proved that the volume spectrum satisfies sublinear growth bounds. 
Namely, there are constants $C_1$ and $C_2$ depending on $M$ such that 
\[
C_1  p^{\frac{1}{n+1}} \le \omega_p \le C_2  p^{\frac{1}{n+1}}.
\]
Later, Liokumovich, Marques, and Neves \cite{LMN} showed that the volume spectrum satisfies a Weyl law. That is, there is a universal constant $a_n$ depending only on the dimension such that 
\[
\omega_p \sim a_n \vol(M)^{\frac{n}{n+1}} p^{\frac{1}{n+1}}
\] 
as $p\to \infty$; see Chodosh-Mantoulidis \cite{CM2} for the calculation of $a_n$ when $n=2$.
The Weyl law for the volume spectrum has been instrumental in the proof of many results on the existence of minimal surfaces in Riemannian manifolds.

In the early 1980s, Almgren \cite{A}, Pitts \cite{P}, and Schoen-Simon \cite{SS} developed a min-max theory for the area functional on closed Riemannian manifolds. Their combined work implies that every closed Riemannian manifold of dimension $3 \le n+1 \le 7$, contains a closed, smooth, embedded minimal surface.  Around the same time, Yau \cite{Y} conjectured that every closed manifold should contain infinitely many minimal surfaces.  Marques and Neves devised a program to prove Yau's conjecture by using the Almgren-Pitts min-max theory to find a minimal surface with area $\omega_p$ for each $p\in \N$.

This program has now been successfully carried out. 
Fix a closed Riemannian manifold $(M^{n+1},g)$ with $3 \le n+1\le 7$.  Irie, Marques, and Neves \cite{IMN} showed that, for a generic metric $g$, the union of all minimal surfaces in $M$ is dense in $M$. In particular, this proved Yau's conjecture for generic metrics.  Later Marques, Neves, and Song \cite{MNS} refined this result to show that, for a generic metric  $g$, there is a sequence of minimal surfaces in $M$ which becomes equidistributed in $M$. The Weyl law for the volume spectrum was a key ingredient in the proof of both of these results. Following the second named author's proof of the Multiplicity One Conjecture \cite{Z}, Marques and Neves \cite{MN2} showed that, for a generic metric $g$, there is a sequence of minimal surfaces $\Sigma_p$ with index $p$ and $\area(\Sigma_p) = \omega_p$. Song \cite{Song} proved Yau's conjecture for arbitrary metrics $g$.

The Almgren-Pitts min-max theory relies heavily on tools from geometric measure theory.  There is a parallel min-max theory for finding minimal surfaces based on the theory of phase transitions.   This theory relies on the Allen-Cahn PDE and the varifold regularity theory of Wickramasekera. Gaspar and Guaraco \cite{GG2} defined a phase transition spectrum $\{c(p)\}_{p\in \N}$ associated to a Riemannian manifold via the Allen-Cahn PDE. They showed that each $c(p)$ is achieved by a collection of minimal surfaces with multiplicities.  Chodosh and Mantoulidis \cite{CM} proved the Multiplicity One Conjecture in the phase transition setting in ambient dimension three.  Thus, for generic metrics on $M^3$, they obtained the existence of a sequence of minimal surfaces $\Sigma_p$ with index $p$ and $\area(\Sigma_p) = c(p)$.

Gaspar and Guaraco \cite{GG} showed that the phase transition spectrum also satisfies a Weyl law. That this, there is a constant $\tau_n$ depending only on the dimension such that 
\[
c(p) \sim \tau_n \vol(M)^{\frac{n}{n+1}}p^{\frac{1}{n+1}}. 
\]
Dey \cite{Dey2} proved that actually $\omega_p = c(p)$ for all $p\in \N$ and thus the Almgren-Pitts volume spectrum and the phase transition volume spectrum coincide. In particular, the constants $a_n$ and $\tau_n$ appearing in the two Weyl laws are equal. 

In this paper, we define a ``half-volume'' spectrum associated to a Riemannian manifold.  In the Almgren-Pitts setting, we restrict to $p$-sweepouts by families of hypersurfaces that each enclose half the volume of $M$.  Then we define 
\[
\tilde \omega_p = \inf_{\text{half-volume $p$-sweepouts }X} \left[\sup_{\Sigma \in X} \area(\Sigma)\right]. 
\] 
The sequence $\{\tilde \omega_p\}_{p\in \N}$ is called the half-volume spectrum of $M$. 
In the phase transition setting, we define an analogous half-volume spectrum $\tilde c(p)$ by looking at critical points of the Allen-Cahn energy subject to the volume constraint $\int_M u = 0$. In both cases, we show that the Weyl law continues to hold. This gives the following theorems. 

\begin{theorem}
\label{wl1}
The Almgren-Pitts half-volume spectrum satisfies 
\[
\tilde \omega_p \sim a_n \vol(M)^{\frac{n}{n+1}}p^{\frac{1}{n+1}}, \quad \text{as } p\to\infty.
\]
\end{theorem} 

\begin{theorem}
\label{wl2}
The phase transition half-volume spectrum satisfies 
\[
\tilde c(p) \sim \tau_n \vol(M)^{\frac{n}{n+1}}p^{\frac{1}{n+1}}, \quad \text{as } p\to\infty.
\]
\end{theorem}

 In the Allen-Cahn setting, we are able to use the results of Bellettini and Wickramasekera \cite{BelWic} to find varifolds achieving each $\tilde c(p)$.  In the following theorem, a hypersurface $\Sigma$ is called almost-embedded if near each point in $M$ either $\Sigma$ is embedded or $\Sigma$ decomposes into an ordered union of embedded sheets.   
 
\begin{theorem}
\label{main}
Let $(M^{n+1},g)$ be a closed Riemannian manifold with $3\le n+1 \le 7$. Fix a number $p\in \N$. There are
\begin{itemize}
\item[(i)] a Caccioppoli set $\Omega\subset M$ with $\vol(\Omega) = \frac 1 2 \vol(M)$ whose boundary is smooth and almost-embedded with constant mean curvature,
\item[(ii)]  a (possibly empty) collection of smooth, disjoint minimal surfaces $\Sigma_1,\hdots,\Sigma_k \subset M \setminus \Omega$, 
\item[(iii)] and positive integers $\theta_0 \in \Z$ and $\theta_1,\hdots,\theta_k\in 2\Z$ 
\end{itemize}
such that
$
\tilde c(p) = \theta_0\area(\bd \Omega) + \theta_1\area(\Sigma_1) + \hdots + \theta_k\area(\Sigma_k).
$
Moreover, $\theta_0 = 1$ unless $\bd \Omega$ is also a minimal surface. 
\end{theorem}

\noindent Note that Theorem 3 produces a constant mean curvature surface that encloses half the volume of $M$.  Previously, the second author and Zhu \cite{ZZ} developed a min-max theory in the Almgren-Pitts setting capable of finding surfaces of constant mean curvature $c$. However, there is no control over the volume enclosed by the surface. Likewise, Bellettini and Wickramasekera \cite{BelWic} developed a min-max theory in the Allen-Cahn setting capable of finding surfaces of constant mean curvature $c$.  Again there is no control over the volume enclosed by the surface.  Thus there is a trade off.  Theorem \ref{main} produces constant mean curvature surfaces enclosing half the volume of $M$, but at the expense of losing control over the exact value of the mean curvature.

We conclude the introduction with some open problems. First, we conjecture that $\tilde \omega_p = \tilde c(p)$ for all $p\in \N$. Second, we conjecture that, for a generic metric $g$, the phase transition half-volume spectrum is achieved by multiplicity one constant mean curvature surfaces enclosing half the volume of $M$. In other words, in Theorem \ref{main} the collection $\Sigma_1,\hdots,\Sigma_k$ is empty and $\theta_0 = 1$ for every $p$.  In particular, we conjecture that generically there are infinitely many constant mean curvature surfaces enclosing half the volume of $M$.  Finally, it is interesting to know whether one can find surfaces achieving $\tilde \omega_p$ by applying the Almgren-Pitts min-max theory with a volume constraint.  This seems to be a difficult task. Already it is not obvious how to define a suitable pull-tight on the space of half-volume cycles. 

\subsection*{Acknowledgement} X.Z. was supported by NSF grant DMS-2243149,  and an Alfred P. Sloan Research Fellowship.

\section{The Almgren-Pitts Half-Volume Spectrum} 
\label{ap} In this section, we investigate the topology of the space of half-volume cycles in a given manifold, and then define the Almgren-Pitts half-volume spectrum.  Let $(M^{n+1},g)$ be a closed Riemannian manifold.  We will use the following notation. 
\begin{itemize}
\item Let $\mathfrak h = \frac{1}{2}\vol M$. 
\item Let $\mathcal C(M)$ denote the collection of all Caccioppoli sets in $M$.  
\item Let $\mathcal C_{\hv}(M)$, $\mathcal C_{\ge \hv}(M)$, and $\mathcal C_{\le \hv}(M)$ denote the space of Caccioppoli sets with volume equal to $\hv$, greater than or equal to $\hv$, and less than or equal to $\hv$, respectively. 
\item Let $\z(M,\Z_2)$ denote the set of all $n$-dimensional flat chains mod 2 in $M$. 
\item Let $\B(M,\Z_2)$ denote the set of all $T\in \z(M,\Z_2)$ such that $T = \bd \Omega$ for some $\Omega\in \C(M)$. This is the connected component of the empty set in $\z(M,\Z_2)$ in the flat topology. 
\item Let $\h(M,\Z_2)$ be the set of all $T \in \z(M,\Z_2)$ such that $T = \bd \Omega$ for some $\Omega\in \C_{\hv}(M)$.  This is the space of ``half-volume cycles.'' 
\item We use $\mathcal F$ to denote the flat topology, $\mathbf F$ to denote the $\mathbf F$-topology, and $\mathbf M$ to denote the mass topology.  All spaces are assumed to be equipped with the flat topology except where otherwise noted. 
\item We will abuse notation and write $\vol(\Omega)$ and $\area(T)$ instead of $\mathbf M(\Omega)$ and $\mathbf M(T)$ for $\Omega\in \mathcal C(M)$ and $T\in \z(M,\Z_2)$, respectively. 
\end{itemize}

\noindent We will show that $\h(M,\Z_2)$ is weakly homotopy equivalent to $\RP^\infty$. The first step is to show that the double cover $\C_\hv(M)$ is contractible. 

\begin{prop}
The space $\C_{\le \hv}(M)$ deformation retracts to $\C_\hv(M)$.
\label{defret}
\end{prop}

\begin{proof}
The union of two Caccioppoli sets is a Caccioppoli set.  Choose a Morse function $f\f M\to \R$.  For $s\in [0,\hv]$, let $B_s$ be the sublevel set of $f$ with volume equal to $s$. Note that each $B_s$ is a Caccioppoli set.  For each $\Omega \in \C_{\le \hv}(M)$ and $t\in [0,1]$, choose a number $s(\Omega,t)$ so that
\[
\vol(\Omega \cup B_{s(\Omega,t)}) = \vol(\Omega) + t(\hv - \vol(\Omega)).
\]
Note that there is not necessarily a unique choice for $s(\Omega,t)$, and the mapping $(\Omega,t)\mapsto s(\Omega,t)$ may not be continuous. Nevertheless, we claim that the map $\phi\f \C_{\le \hv}(M)\times [0,1]\to \C_{\le \hv}(M)$ defined by 
\[
\phi(\Omega,t) = \Omega \cup B_{s(\Omega,t)}
\]
is continuous in the flat topology. Given this claim, $\phi$ is the required deformation retraction.  Indeed, $\phi(\Omega,t) = \Omega$ for all $\Omega\in \C_\hv(M)$ and all $t\in [0,1]$, and moreover, $\phi(\Omega,1) \in \mathcal C_\hv(M)$ for all $\Omega\in \C_{\le \hv}(M)$.

To see that $\phi$ is continuous, let $\eps,\eta > 0$ be small positive numbers. Assume that $\Omega,\Theta \in \C_{\le \hv}(M)$ satisfy $\vol(\Omega \sd \Theta) < \eps$ and that $t,r\in [0,1]$ satisfy $\vert t-r\vert < \eta$. We need to check that 
$
\vol\big(\phi(\Omega,t)\operatorname{\Delta} \phi(\Theta,r)\big)
$
is small. Let $U = B_{s(\Omega,t)}\setminus \Omega$ and let $V = B_{s(\Theta,r)}\setminus \Theta$.  Now observe that 
\begin{align*}
\vol\big(\phi(\Omega,t)\operatorname{\Delta} \phi(\Theta,r)\big) &= \vol\big([\Omega\cup U] \sd [\Theta\cup V]\big)\\ &= \vol(\Omega) + \vol(U) + \vol(\Theta) + \vol(V) \\
&\quad - 2\vol(\Omega\cap \Theta) - 2\vol(\Omega \cap V) - 2 \vol(U\cap \Theta) - 2\vol(U\cap V).
\end{align*}
Note that $\vol(\Omega) + \vol(\Theta) - 2\vol(\Omega\cap \Theta) =  \vol(\Omega \sd \Theta) < \eps$. We also have $\vol(\Omega \cap V) \le \vol(\Omega\setminus \Theta) < \eps$ and $\vol(U\cap \Theta) \le \vol(\Theta\setminus \Omega) < \eps$. 

Therefore, it remains to show that 
$
\vol(U) + \vol(V) - 2\vol(U\cap V)
$
is small.  Observe that $\vol(U) - \vol(U\cap V) = \vol(U\setminus V)$ and that $\vol(V) - \vol(U\cap V) = \vol(V\setminus U)$.  Without loss of generality, we can assume that $s(\Theta,r) \ge s(\Omega,t)$.  In this case, we have 
\begin{align*}
\vol(U\setminus V) &= \vol\big([B_{s(\Omega,t)}\setminus \Omega] \setminus [B_{s(\Theta,r)}\setminus \Theta]\big)\\
&\le \vol\big([B_{s(\Theta,r)}\setminus \Omega] \setminus [B_{s(\Theta,r)}\setminus \Theta]\big) \le \vol(\Theta \setminus \Omega) < \eps. 
\end{align*}
We also have 
\begin{align*}
\vol(V\setminus U) &= \vol\big([B_{s(\Theta,r)}\setminus \Theta] \setminus [B_{s(\Omega,t)}\setminus \Omega]\big)\\
&= \vol\big([B_{s(\Theta,r)}\setminus B_{s(\Omega,t)}]\setminus \Theta\big)+ \vol\big([B_{s(\Omega,t)}\setminus \Theta] \setminus [B_{s(\Omega,t)}\setminus \Omega]\big)\\
&\le \vol\big([B_{s(\Theta,r)}\setminus B_{s(\Omega,t)}]\setminus \Theta\big) + \vol(\Omega\setminus \Theta) \\
&< \vol\big([B_{s(\Theta,r)}\setminus B_{s(\Omega,t)}]\setminus \Theta\big) + \eps. 
\end{align*}
Thus to prove the claim, it remains to show that $\vol\big([B_{s(\Theta,r)}\setminus B_{s(\Omega,t)}]\setminus \Theta\big)$ is small.  Here we must use the assumption that $t$ is close to $r$. Notice that  
\begin{gather*}
r(\hv - \vol(\Theta)) = \vol(V) = \vol\big([B_{s(\Theta,r)}\setminus B_{s(\Omega,t)}]\setminus \Theta\big) + \vol(B_{s(\Omega,t)}\setminus \Theta), \\
t(\hv-\vol(\Omega)) = \vol(U) = \vol(B_{s(\Omega,t)}\setminus \Omega). 
\end{gather*}
Now observe that 
\begin{align*}
\vert r(\hv - \vol(\Theta)) - t(\hv-\vol(\Omega))\vert &\le  \vert t - r\vert \hv + t\vert \vol(\Omega)-\vol(\Theta)\vert+ \vert t-r\vert \vol(\Theta)\\
&\le 2\eta \hv + \eps. 
\end{align*}
Also we have 
\begin{align*}
\vert \vol(B_{s(\Omega,t)}\setminus \Theta) - \vol(B_{s(\Omega,t)}\setminus \Omega)\vert  &\le \vol \big([B_{s(\Omega,t)}\setminus \Theta] \sd [B_{s(\Omega,t)}\setminus \Omega]\big)\\
&\le \vol(\Theta\setminus \Omega) + \vol(\Omega\setminus \Theta) < \eps.
\end{align*}
This implies that 
\[
\vol\big([B_{s(\Theta,r)}\setminus B_{s(\Omega,t)}]\setminus \Theta\big) \le 2\eta \hv + 2\eps,
\]
which completes the proof of the claim. 
\end{proof}

\begin{prop}
The space $\C(M)$ deformation retracts to $\C_{\hv}(M)$, and the space $\B(M,\Z_2)$ deformation retracts to $\h(M,\Z_2)$.
\label{defret2}
\end{prop}

\begin{proof}
Consider the deformation retraction $\phi\f \C_{\le \hv}(M)\times [0,1] \to \C_{\le \hv}(M)$ from the previous proposition. We can extend $\phi$ to an odd map $\psi\f \C(M) \times [0,1]\to \C(M)$ by the formula 
\[
\psi(\Omega,t) = \begin{cases}
\phi(\Omega,t), &\text{if } \vol(\Omega) \le \hv\\
M\setminus \phi(M\setminus \Omega,t), &\text{if } \vol(\Omega) \ge \hv.
\end{cases}
\]
Then $\psi$ is a deformation retraction of $\C(M)$ onto $\C_{\hv}(M)$.  Moreover, since $\psi$ is odd, this descends to a map $\theta\f \B(M,\Z_2) \times [0,1]\to \B(M,\Z_2)$.
This is the required deformation retraction of $\B(M,\Z_2)$ onto $\h(M,\Z_2)$. 
\end{proof}

\begin{prop}
\label{bound}
Let $K$ be the maximal area of a level set of the Morse function $f$ used in the proof of Proposition \ref{defret}.  Let $\theta$ be the deformation retraction from Proposition \ref{defret2}. Then $\area(\theta(T,t)) \le \area(T)+K$ for all $T\in \B(M,\Z_2)$ and all $t\in[0,1]$. 
\end{prop}

\begin{proof}
Fix some $T\in \B(M,\Z_2)$ and some $t\in [0,1]$. Choose a set $\Omega\in \C_{\le \hv}(M)$ such that $\bd \Omega = T$. Let $\phi$ be the deformation retraction from Proposition \ref{defret}. Then 
\[
\theta(T,t) = \bd \phi(\Omega,t) = \bd(\Omega \cup B_{s(\Omega,t)}). 
\]
Note that $\bd(\Omega \cup B_{s(\Omega,t)}) \subset \bd \Omega \cup \bd B_{s(\Omega,t)}$ and therefore $\area(\theta(T,t)) \le \area(\bd \Omega) + K = \area(T) + K$, as needed
\end{proof}

The homotopy groups of the cycle spaces were originally computed by Almgren \cite{A1}.
Later, Marques and Neves \cite{MN2} gave a simplified proof in the case of codimension 1 cycles.

\begin{theorem}[Marques-Neves]
\label{topz}
The map $\bd \f \mathcal C(M)\to \B(M,\Z_2)$ is a double cover. The space $\mathcal C(M)$ is contractible, and $\B(M,\Z_2)$ is weakly homotopy equivalent to $\RP^\infty$.
\end{theorem} 

Combined with the previous propositions, this yields the following corollary. 

\begin{corollary}
\label{toph}
The map $\bd \f \C_{\hv}(M)\to \h(M,\Z_2)$ is a double cover. The space $\C_{\hv}(M)$ is contractible and $\h(M,\Z_2)$ is weakly homotopy equivalent to $\RP^\infty$. The inclusion map $\h(M,\Z_2) \to \B(M,\Z_2)$ is a homotopy equivalence. 
\end{corollary}

We now recall the notion of sweepouts.  Since $\B(M,\Z_2)$ is weakly homotopy equivalent to $\RP^\infty$,  it follows that the cohomology ring of $\B(M,\Z_2)$ with $\Z_2$ coefficients is $\Z_2[\lambda]$, where the generator $\lambda$ is of degree 1.  Let $X$ be a cubical complex. 

\begin{defn}
\label{ps}
A flat continuous map $\Phi\f X\to \B(M,\Z_2)$ is called a $p$-sweepout if $\Phi^*\lambda^p \neq 0$ in $H^p(X,\Z_2)$. 
\end{defn}

\begin{defn}
A map $\Phi\f X\to \z(M,\Z_2)$ is said to have {\it no concentration of mass} provided 
\[
\lim_{r\to 0} \left[\sup_{q\in M} \sup_{x\in X} \area(\Phi(x) \llcorner B(q,r))\right] = 0.
\]
\end{defn}

\begin{defn} Let $\mathcal P_p(M)$ denote the collection of all $p$-sweepouts of $M$ with no concentration of mass.  Note that different $p$-sweepouts may have different domains. 
\end{defn}

\begin{defn} The $p$-width of $M$ is 
\[
\omega_p = \inf_{\Phi\in \mathcal P_p(M)} \left[\sup_{x\in \operatorname{dom}(\Phi)} \area(\Phi(x))\right]. 
\]
\end{defn}

\begin{rem}
In \cite{LMN}, the authors state that the cohomology ring of $\z(M,\Z_2)$ is isomorphic to $\Z_2[\lambda]$.  Then they  define a $p$-sweepout as a map $\Phi\f X\to \z(M,\Z_2)$ such that $\Phi^*\lambda^p\neq 0$. However, the cohomology ring of $\z(M,\Z_2)$ is actually $\oplus_i \Z_2[\lambda_i]$ where the direct sum is taken over the connected components of $\z(M,\Z_2)$. These connected components are in bijection with homology classes in $H_n(M,\Z_2)$.  
Given this, there are several possible ways to define a $p$-sweepout.  The simplest, which we shall adopt, is to replace the space $\z(M,\Z_2)$ with $\B(M,\Z_2)$ as in Definition \ref{ps} so that the cohomology ring is indeed $\Z_2[\lambda]$. Alternatively, one could define a $p$-sweepout as a map $\Phi\f X\to \z(M,\Z_2)$ such that $\Phi^*\lambda_i^p\neq 0$ for some $i$. In either case, it is straightforward to see that one still obtains a Weyl law for the resulting $p$-widths. 
\end{rem} 

We can now introduce the central object of the paper.  By Corollary \ref{toph}, the cohomology ring of $\h(M,\Z_2)$ with $\Z_2$ coefficients is also $\Z_2[\lambda]$.  Again let $X$ be a cubical complex.

\begin{defn} 
A flat continuous map $\Phi\f X\to \h(M,\Z_2)$ is called a half-volume $p$-sweepout if $\Phi^* \lambda^p \neq 0$ in $H^p(X,\Z_2)$. 
\end{defn} 

\begin{defn}
Let $\mathcal Q_p(M)$ denote the collection of all half-volume $p$-sweepouts of $M$ with no concentration of mass. 
\end{defn}

\begin{defn}
The half-volume $p$-width of $M$ is 
\[
\tilde \omega_p = \inf_{\Phi\in \mathcal Q_p(M)}\left[\sup_{x\in \operatorname{dom}(\Phi)} \area(\Phi(x))\right]. 
\]
\end{defn}

\noindent We will call the sequence $\{\tilde \omega_p\}_{p\in \N}$ the half-volume spectrum of $M$.

Liokumovich, Marques, and Neves \cite{LMN} showed that the $p$-widths of $M$ satisfy a Weyl law. 

\begin{theorem}[Liokumovich, Marques, Neves]
\label{weyl} 
There is a universal constant $a_n$ such that 
$
\omega_p \sim a_n \vol(M)^{n/(n+1)}p^{1/(n+1)}$ as $p\to \infty$.
\end{theorem}

Next, we will show that the half-volume spectrum also satisfies a Weyl law.  It is possible to prove this directly.  
  However, this is not the approach we will take.  Rather, we will show that the Weyl law for the half-volume spectrum follows from Theorem \ref{weyl}, together with the fact that every $p$-sweepout is homotopic to a $p$-sweepout by half-volume cycles.

\begin{prop}
\label{comp1}
The half-volume spectrum satisfies $\omega_p \le \tilde \omega_p$ for all $p\in \N$. 
\end{prop}

\begin{proof}
Notice that any half-volume $p$-sweepout with no concentration of mass automatically belongs to ${\mathcal P_p}(M)$.  Therefore, the proposition follows immediately from the definitions of $\omega_p$ and $\tilde \omega_p$. 
\end{proof}

\begin{prop}
\label{comp2}
There is a constant $K$ depending only on $M$ such that $\tilde \omega_p \le  \omega_p + K + 1$ for all $p \in \N$. 
\end{prop}

\begin{proof}
Choose a $p$-sweepout $\Phi\f X\to \B(M,\Z_2)$ in ${\mathcal P_p}(M)$ with 
\[
\sup_{x\in X}\area(\Phi(x)) \le  \omega_p + 1.
\]
Let $\theta\f \B(M,\Z_2)\times[0,1]\to \B(M,\Z_2)$ be the deformation retraction constructed in Proposition \ref{defret2}. By Proposition \ref{bound}, there is a constant $K$ such that 
\[
\area(\theta(T,t)) \le \area(T) + K
\]
for all $T \in \B(M,\Z_2)$ and all $t\in[0,1]$.  Therefore, the map $\Psi\f X \to \h(M,\Z_2)$ given by 
$
\Psi(x) = \theta(\Phi(x),1)
$
is a half-volume $p$-sweepout with
\[
\sup_{x\in X} \area(\Psi(x)) \le \sup_{x\in X} \area(\Phi(x)) + K.
\]
Moreover, it is straightforward to check that $\Psi$ has no concentration of mass.  This proves that $\tilde \omega_p \le  \omega_p + K + 1$. 
\end{proof} 

We are now able to prove Theorem \ref{wl1}.

\begin{theo}
The Weyl law holds for the half-volume spectrum.  In other words, we have $\tilde \omega_p \sim a_n \vol(M)^{n/(n+1)}p^{1/(n+1)}$ as $p\to \infty$. 
\end{theo}

\begin{proof}
This follows from Proposition \ref{comp1} and Proposition \ref{comp2}.  Indeed, we have 
\[
\omega_p \le \tilde \omega_p \le \omega_p + K + 1.
\]
Theorem \ref{weyl} implies that 
\[
\lim_{p\to\infty} \frac{\omega_p}{a_n \vol(M)^{n/(n+1)}p^{1/(n+1)}} = 1 \quad \text{and} \quad \lim_{p\to\infty} \frac{\omega_p+ K + 1}{a_n \vol(M)^{n/(n+1)}p^{1/(n+1)}} = 1,
\]
and it follows that 
\[
\lim_{p\to\infty} \frac{\tilde \omega_p}{a_n \vol(M)^{n/(n+1)}p^{1/(n+1)}} = 1
\]
as well.
\end{proof}

\section{The Phase Transition Half-Volume Spectrum} 

There is also an analogous half-volume spectrum in the Allen-Cahn setting.  Let $W\f \R\to \R$ be an even double-well potential.  This means that 
\begin{itemize}
\item[(i)] $W$ is smooth and non-negative,
\item[(ii)] $W(x) = W(-x)$ for all $x\in \R$,
\item[(iii)] $W$ has non-degenerate minima $W(\pm1)=0$, \label{property(iii)}
\item[(iv)] $W$ has a non-degenerate maximum $W(0)>0$, 
\item[(v)] $W$ is increasing on $(-1,0)$ and $(1,\infty)$ and decreasing on $(0,1)$ and $(-\infty,-1)$,
\item[(vi)] there are constants $\kappa > 0$ and $\alpha\in(0,1)$ such that $W''(x) \ge \kappa$ for all $\vert x\vert \ge \alpha$. 
\end{itemize}
Define the constant 
\[
\sigma = \int_{-1}^1 \sqrt{W(s)/2}\, ds.
\]
Let $u\f M\to \R$ be an $W^{1,2}$ function.  For $\eps > 0$  define the Allen-Cahn energy 
\[
E_\eps(u) = \int_M \frac \eps 2 \vert  \grad u\vert^2 + \frac{W(u)}{\eps}.
\]
In \cite{GG2}, Gaspar and Guaraco define a phase-transition spectrum associated to $M$ via the Allen-Cahn energy.

In order to state the definition of the spectrum, we shall need some further background. A paracompact topological space $X$ is called a $\Z_2$-space if it admits a free $\Z_2$-action. Given such a space, there is always a quotient space $T = X/\Z_2$ and the natural map $X\to T$ is a principal $\Z_2$-bundle. Any such bundle arises as a pullback of the universal bundle $S^\infty \to \RP^\infty$. More precisely, there is a classifying map $f\f T\to \RP^\infty$ such that $X\to T$ is the pullback of $S^\infty \to \RP^\infty$ via $f$. The Alexander-Spanier cohomology ring of $\RP^\infty$ with $\Z_2$ coefficients is $\Z_2[\mu]$ where the generator $\mu$ is in degree one. The map $f$ is unique up to homotopy, and therefore the cohomology classes $f^*\mu^p$ are well-defined in the Alexander-Spanier cohomology ring $H^*(T,\Z_2)$. The $\Z_2$-index of $X$ is defined to be the largest $p$ such that $f^*\mu^{p-1} \neq 0$ in $H^*(T,\Z_2)$. A subspace $A$ of $X$ is called invariant if it is closed under the $\Z_2$-action. 

The $\Z_2$-index enjoys the following useful properties.  See Fadell and Rabinowitz \cite{FR} for more details.  
\begin{itemize}
\item[(i)] (Monotonicity) If $X_1$ and $X_2$ are $\Z_2$-spaces and there is a continuous equivariant map $X_1\to X_2$ then $\ind_{\Z_2}(X_1) \le \ind_{\Z_2}(X_2)$. 

\item[(ii)] (Subadditivity) If $X$ is a $\Z_2$-space and $A_1$ and $A_2$ are closed, invariant subsets with $A_1\cup A_2 = X$ then $\ind_{\Z_2}(X) \le \ind_{\Z_2}(A_1) + \ind_{\Z_2}(A_2)$. 

\item[(iii)] (Continuity) If $X$ is a $\Z_2$-space and $A$ is a closed, invariant subset of $X$ then there is an invariant neighborhood $V$ of $A$ in $X$ such that $\ind_{\Z_2}(A) = \ind_{\Z_2}(\cl V)$. 
\end{itemize}

The space $W^{1,2}(M)\setminus \{0\}$ is paracompact since it is a metric space. Moreover, it admits a natural $\Z_2$ action $u\mapsto -u$.  Note that $E_\eps$ respects this action since $E_\eps(u) = E_\eps(-u)$.  This uses the fact that $W$ is even. A set $A\subset W^{1,2}(M)\setminus\{0\}$ is called invariant provided $u\in A$ if and only if $-u\in A$.  The $\Z_2$-action on $W^{1,2}(M)\setminus \{0\}$ descends to any such $A$.   Define the families 
\[
\mathcal F_p = \{A\subset W^{1,2}(M)\setminus\{0\}:\ A\text{ is compact and invariant with } \ind_{\Z_2}(A) \ge p+1\}.
\]
Gaspar and Guaraco define the min-max values 
\[
c(\eps,p) = \frac{1}{2\sigma} \inf_{A\in \mathcal F_p} \left[\sup_{u\in A} E_\eps(u)\right]. 
\]
Then they set $c(p) = \liminf_{\eps\to 0} c(\eps,p)$. The sequence $\{c(p)\}_{p\in \N}$ is the phase transition spectrum of $M$.  

Gaspar and Guaraco \cite{GG} showed that the Weyl law also holds for the phase transition spectrum. 

\begin{theorem}[Gaspar and Guaraco]
There is a universal constant $\tau_n$ such that $c(p)\sim \tau_n \vol(M)^{n/(n+1)} p^{1/(n+1)}$ as $p\to \infty$. 
\end{theorem}

\noindent Dey \cite{Dey2} proved that $\omega_p = c(p)$ for all $p\in \N$. In particular, this implies that the constant $\tau_n$ is equal to the constant $a_n$.

\begin{rem}
Gaspar and Guaraco do not include the normalization constant $\frac{1}{2\sigma}$ in the definition of $c(\eps,p)$ and $c(p)$. We have chosen to include it so that one has $\omega_p = c(p)$. 
\end{rem}

It is also possible to define a half-volume spectrum in the phase transition setting.  Define 
\[
Y = \{u\in W^{1,2}(M): \int_M u = 0\}. 
\]
Note that $Y$ is a closed subspace of $W^{1,2}(M)$ and so $Y$ is also a Hilbert space.  
We can run essentially the same construction using $Y$ in place of $W^{1,2}(M)$. For each $p\in \N$, define 
\[
\mathcal G_p = \{A\subset Y\setminus\{0\}:\, A \text{ is compact and invariant with } \ind_{\Z_2}(A) \ge p+1\},
\]
and then set 
\[
\tilde c(\eps,p) = \frac{1}{2\sigma} \inf_{A\in \mathcal G_p}\left[\sup_{u\in A} E_\eps(u)\right].
\]
Taking the limit as $\eps \to 0$ gives the phase-transition half volume spectrum. 

\begin{defn} For each $p\in \N$, let $\tilde c(p) = \liminf_{\eps\to 0} \tilde c(\eps,p)$.  The phase transition half volume spectrum of $M$ is the sequence $\{\tilde c(p)\}_{p\in \N}$. 
\end{defn} 

\begin{prop} 
\label{Comp1}
The phase transition half-volume spectrum satisfies $c(p) \le \tilde c(p)$ for all $p\in \N$. 
\end{prop}

\begin{proof}
Note that $\mathcal G_p \subset \mathcal F_p$ for every $p\in \N$. Therefore, for every $\eps > 0$, it holds that $c(\eps,p)\le \tilde c(\eps,p)$. The result then follows by sending $\eps \to 0$. 
\end{proof}

\begin{prop}
\label{Comp2}
The phase transition half-volume spectrum satisfies $\tilde c(p) \le c(p+1)$ for all $p \in \N$. 
\end{prop}

\begin{proof}
Fix an $\eps > 0$. Select a set $A\in \mathcal F_{p+1}$ with 
\[
\sup_{u\in A} E_\eps(u) \le 2\sigma\left[ c(\eps,p+1)+ \eps\right].
\]
Define the set $B = \{u\in A: \int_M u = 0\}$ and note that $B$ is closed and invariant. We claim that $\ind_{\Z_2}(B) \ge p+1$ so that $B \in \mathcal G_p$. Given this, we obtain that $\tilde c(\eps,p) \le c(\eps,p+1)+\eps$, and the result follows upon sending $\eps \to 0$. 

It remains to prove the claim.  By the continuity of the index, there is a neighborhood $V$ of $B$ in $A$ such that $\ind_{\Z_2}(B) = \ind_{\Z_2}(\cl V)$. There is an $\eta > 0$ such that 
\[
\{u\in A: -\eta < \int_M u < \eta\} \subset V.
\]
Indeed, if not, then there is a sequence $u_k$ in $A\setminus V$ with $\int_M u_k\to 0$. Since $A\setminus V$ is compact, we can find a subsequence $u_{k_j}$ that converges to a limit $u$ in $A\setminus V$. But $u$ satisfies $\int_M u = 0$ and therefore $u \in B\subset V$ and this is a contradiction. Therefore, such an $\eta$ exists. 

Let $K = \{u\in A:\, \left\vert \int_M u \right\vert \ge \frac \eta 2\}$. Then $K$ is a closed invariant subset of $A$ and $K \cup \cl V = A$.  Define a map $K\to S^0$ by sending $u$ to $1$ if $\int_M u > 0$ and sending $u$ to $-1$ if $\int_M u < 0$. This map is continuous and equivariant and so by the monotonicity of the index we have 
\[
\ind_{\Z_2}(K) \le \ind_{\Z_2}(S^0) = 1. 
\]
Hence by the subadditivity of the index, we get 
\[
p+2 \le \ind_{\Z_2}(A) \le \ind_{\Z_2}(K) + \ind_{\Z_2}(\cl V) \le \ind_{\Z_2}(\cl V) + 1. 
\]
This implies that $\ind_{\Z_2}(B) = \ind_{\Z_2}(\cl V) \ge p+1$, and so $B \in \mathcal G_p$ as needed. 
\end{proof}

We can now prove Theorem \ref{wl2}. 

\begin{theo}
The phase transition half-volume spectrum satisfies the Weyl law. In other words, we have $\tilde c(p) \sim \tau_n \vol(M)^{n/(n+1)} p^{1/(n+1)}$ as $p\to \infty$. 
\end{theo}

\begin{proof}
By Propositions \ref{Comp1} and \ref{Comp2} we have 
\[
c(p) \le \tilde c(p) \le c(p+1). 
\]
By the Weyl law for the phase-transition spectrum, we have
\[
\lim_{p\to \infty} \frac{c(p)}{\tau_n \vol(M)^{n/(n+1)}p^{1/(n+1)}} = 1 \quad \text{ and } \quad \lim_{p\to \infty} \frac{c(p+1)}{\tau_n \vol(M)^{n/(n+1)}p^{1/(n+1)}} = 1 
\]
and therefore
\[
\lim_{p\to \infty} \frac{\tilde c(p)}{\tau_n \vol(M)^{n/(n+1)}p^{1/(n+1)}} = 1 
\]
as well. 
\end{proof}

\section{Surfaces Associated to the Half-Volume Spectrum}

In this section, we use the Allen-Cahn min-max theory to construct surfaces associated to the phase transition half-volume spectrum. The goal is to prove Theorem \ref{main}. Fix a closed Riemannian manifold $(M^{n+1},g)$ with $3\le n+1\le 7$. 
Fix a number $p\in \N$. In this section, we require the following additional hypothesis on the double-well potential $W$.
\begin{itemize}
\item[(vii)] There are constants $0 < C_1 < C_2$ and $\beta > 1$ and $2 < q < \frac{11}{5}$ such that 
\[
C_1 \vert x\vert^q \le W(x) \le C_2 \vert x\vert^q \quad \text{and} \quad C_1 \vert x\vert^{q-1} \le \vert W'(x)\vert \le C_2 \vert x\vert^{q-1}
\]
for all $\vert x\vert \ge \beta$. 
\end{itemize}
The first step of the proof is to construct, for each small enough $\eps > 0$, a critical point $u_\eps$ of $E_\eps$ subject to the volume constraint 
\[
\int_M u_\eps = 0.
\]
Given such a $u_\eps$, there is a Lagrange multiplier $\lambda_\eps\in \R$ such that $u_\eps$ is a critical point of 
\[
 F_{\eps,\lambda_\eps}(v) = E_\eps(v) + \lambda_\eps \int_M v
\]
on all of $W^{1,2}$. The construction of $u_\eps$ is similar to that of Gaspar and Guaraco \cite{GG2} in the unconstrained case. 

\begin{rem}
There are two purposes for imposing the growth condition (vii).  The first is that it allows us to verify the Palais-Smale condition with the volume constraint.  In the unconstrained case, one has 
\[
E_\eps(\max(\min(u,1),-1))\le E_\eps(u)
\]
and so by a truncation argument it is enough to verify the Palais-Smale condition along Palais-Smale sequences which are bounded in $L^\infty$.  See \cite{GG2} for more details.  However, truncation may not preserve the volume constraint.  In the volume constrained case, we instead rely on (vii) to show that $W'(u) \in L^2$ whenever $u \in W^{1,2}$.   The second purpose is to get uniform $L^\infty$ bounds on critical points of $F_{\eps,\lambda}$.  Given a sequence $\eps_k\to 0$ and critical points $u_{\eps_k}$ of $F_{\eps_k,\lambda_{\eps_k}}$ with $E_{\eps_k}(u_{\eps_k})\le C$, the growth condition (vii) implies that $\|u_{\eps_k}\|_{L^\infty} \le \beta$ provided $k$ is large enough. 
\end{rem}

We recall (see Proposition 4.4 in \cite{Gua}) that the first variation of $E_\eps$ is given by 
\[
DE_\eps(u)(v) = \int_M \frac \eps 2 \grad u \cdot \grad v + \frac{W'(u)}{\eps}v.
\]
Fix a number $\eps > 0$.  A sequence $A_k$ in $\mathcal G_p$ is called a critical sequence if 
\[
\lim_{k\to \infty} \left[\sup_{u\in A_k} E_\eps(u)\right] = 2\sigma \tilde c(\eps,p). 
\]
A sequence $u_k \in A_k$ is called a min-max sequence provided 
$
\lim_{k\to \infty} E_\eps(u_k) = 2\sigma \tilde c(\eps,p).
$

In the unconstrained case, it is not necessarily true that every min-max sequence is bounded in $W^{1,2}$.  However, one can obtain the existence of a bounded min-max sequence via a truncation argument.  See, for example, the remarks before Proposition 4.5 in \cite{Gua}.  We cannot employ truncation because it doesn't preserve the volume constraint.  Fortunately, in the volume constrained case, every min-max sequence is automatically bounded in $W^{1,2}$. 

\begin{prop}
\label{mmb}
Any min-max sequence $u_k$ is uniformly bounded in $W^{1,2}(M)$.
\end{prop}

\begin{proof}
Assume that $u\in Y$ satisfies $E_\eps(u) \le K$. Since $W\ge 0$, it follows immediately that 
\[
\int_M \vert \grad u\vert^2 \le \frac{2 K}{\eps}. 
\]
Since $u$ has average $0$, the Poincare inequality implies that $\|u\|_{W^{1,2}} \le CK/\eps$.  This proves the result. 
\end{proof}

\begin{prop}
Assume that $u_k$ is a sequence uniformly bounded in $W^{1,2}$. Then $W'(u_k)$ is uniformly bounded in $L^2$. 
\end{prop}

\begin{proof}
By assumption the sequence $u_k$ is uniformly bounded in $W^{1,2}$.  As $3\le n+1 \le 7$, the Sobolev embedding theorem implies that $u_k$ is uniformly bounded in $L^{12/5}$. Now $\vert W'(u_k)\vert \le C \vert u_k\vert ^{q-1}\le C \vert u_k\vert^{6/5}$ whenever $\vert u_k\vert \ge \beta$. Therefore 
\begin{align*}
\int_M W'(u_k)^2 &= \int_{\vert u_k\vert\le \beta} W'(u_k)^2 + \int_{\vert u_k\vert > \beta} W'(u_k)^2\\
&\le C\vol(M) + C \int_M \vert u_k\vert^{12/5},
\end{align*}
and it follows that $W'(u_k)$ is uniformly bounded in $L^2$. 
\end{proof}

 The functional $E_\eps|_Y$ satisfies the Palais-Smale condition. See \cite{Gua} Proposition 4.4 for the proof without a volume constraint. 
 
\begin{prop}
The functional $E_\eps|_Y$ satisfies the Palais-Smale condition.  More precisely, assume that $u_k$ is a bounded sequence in $Y$  and that $\|DE_\eps|_Y(u_k)\| \to 0$. Then a subsequence of $u_k$ converges strongly to a limit $u \in Y$. 
\end{prop}

\begin{proof}
Assume that $u_k$ is a bounded sequence in $Y$ such that $\|DE_\eps|_Y(u_k)\|\to 0$. We need to show that some subsequence of $u_k$ converges strongly to a limit $u\in Y$. Note that $Y$ is closed and convex in $W^{1,2}(M)$ and so $Y$ is weakly closed. Thus, passing to a subsequence, we can assume that $u_k$ converges weakly in $W^{1,2}$ and strongly in $L^{12/5}$ to a point $u\in Y$.   

Observe that 
\begin{align*}
DE_\eps|_Y(u)(u_k-u) = \int_M \eps \grad u \cdot \grad (u_k-u) + \int_M \frac{W'(u)}{\eps} (u_k-u). 
\end{align*}
The first term on the right hand side goes to 0 by the weak convergence $u_k \weak u$.  Note that $W'(u) \in L^2$ since $u\in L^{12/5}$.  Therefore the second term on the right hand side also goes to 0 since $u_k \to u$ in $L^2$.  Thus we obtain 
\[
DE_\eps|_Y(u)(u_k-u) \to 0, \quad \text{as } k\to\infty.
\]
Also note that $DE_\eps|_Y(u_k)(u_k-u) \to 0$ since $\|DE_\eps|_Y(u_k)\|\to 0$ and $u_k-u$ is uniformly bounded in $W^{1,2}$.  On the other hand, 
\[
DE_\eps|_Y(u_k)(u_k-u) = \int_M \eps \grad u_k\cdot \grad(u_k-u) + \int_M \frac{W'(u_k)}{\eps}(u_k-u). 
\]
The second term on the right hand side goes to 0 as $W'(u_k)$ is uniformly bounded in $L^2$ and $u_k-u\to 0$ in $L^2$. 

Now observe that
\begin{align*}
&DE_\eps|_Y(u_k)(u_k-u) - DE_\eps|_Y(u)(u_k-u) \\
&\qquad = \int_M \eps \vert \grad u_k - \grad u\vert^2 + \int_M \frac{W'(u_k)}{\eps}(u_k-u) - \int_M \frac{W'(u)}{\eps}(u_k-u).
\end{align*}
We have already seen that every term in this formula goes to 0 except $\int_M \eps \vert \grad u_k - \grad u\vert^2$, and therefore $\int_M \eps \vert \grad u_k - \grad u\vert^2$ goes to 0 as well. This proves that $u_k \to u$ strongly in $W^{1,2}$, as needed. 
\end{proof}
 
 According to Gaspar and Guaraco \cite{GG2}, for each given $p$, we have $2\sigma c(\eps,p+1) < E_\eps(0)$ provided $\eps$ is small enough.  Therefore, we also have $2 \sigma \tilde c(\eps,p) < E_\eps(0)$ provided $\eps$ is small enough.  Hence, for $\eps$ small enough, any min-max sequence remains bounded away from 0. By the classical theory for functionals satisfying the Palais-Smale condition (see \cite{Stru}), we get the following existence result for critical points of $E_\eps|_Y$.  See Theorem 3.3 in \cite{GG2} for the corresponding result in the unconstrained case. 

\begin{prop} 
\label{ace} Fix $p \in \N$.  For all small enough $\eps$, there is a critical point $u_\eps \in Y$ of $E_\eps|_Y$ with $E_\eps(u_\eps) = \tilde c(\eps,p)$.   There is a number $\lambda_\eps \in \R$ such that $u_\eps$ is a critical point of 
\[
v\mapsto F_{\eps,\lambda_\eps}(v) = E_\eps(v) + \lambda_\eps \int_M v
\]
on all of $W^{1,2}$, and $u_\eps$ satisfies the PDE 
\[
- \eps  \lap u_\eps + \frac{W'(u_\eps)}{\eps} = \lambda_\eps
\]
in the weak sense. Moreover, we have $\int_M u_\eps = 0$. The index of $u_\eps$ as a critical point of $E_\eps|_ Y$ is at most $p$.  
\end{prop}

Given the existence of $u_\eps$, the second step in the proof is to study the convergence of $u_\eps$ as $\eps \to 0$.  Fortunately for us, Bellettini and Wickramasekera \cite{BelWic} have already studied the regularity of such limits.  Let us recall the setup in \cite{BelWic}.   For each $\eps > 0$, suppose $u_{\eps}$ is a critical point of $F_{\eps,\lambda_{\eps}}$ and that $E_\eps(u_\eps)$ and $\ind_{F_{\eps,\lambda_\eps}}(u_\eps)$ and $\|u_{\eps_k}\|_{L^\infty}$ are uniformly bounded.  Choose a sequence $\eps_j\to 0$ and assume that $\lambda_{\eps_j} \to \lambda$.  Passing to a subsequence if necessary, there exist a radon measure $\mu$ on $M$ and a function $u_\infty \in BV(M)$ such that 
\[
\frac{1}{2\sigma}\left(\frac{\eps_j}{2}\vert \grad u_{\eps_j}\vert^2 + \frac{W(u_{\eps_j})}{\eps_j}\right) \weak \mu
\]
and $u_{\eps_j} \to u_\infty$ in $L^1$. Moreover, $u_\infty$ takes only the values $\pm 1$. Hutchinson and Tonegawa \cite{HT} proved that there is an integral varifold $V$ on $M$ such that $\|V\| = \mu$.  The following is the special case of Theorem 4.1 in Bellettini and Wickramasekera \cite{BelWic} where the prescription functions are assumed to be constants and the ambient dimension is assumed to be between 3 and 7. 

\begin{theorem}[See Theorem 4.1 in \cite{BelWic}]
\label{bw}
Let $(M^{n+1},g)$ be a closed Riemannian manifold with $3\le n+1\le 7$.  Assume $(u_{\eps_j})$ is a sequence as above, and assume that $\lambda > 0$. Let $\Omega = \operatorname{int}(\{x\in M:\, u_\infty(x)= 1\})$.   Then $V = V_0 + V_\lambda$ where 
\begin{itemize}
\item[(i)] $V_0$ is induced by a collection of smooth, disjoint minimal surfaces equipped with even multiplicities.  Moreover, $\operatorname{spt}(\|V_0\|) \subset M\setminus \Omega$. 
\item[(ii)] If $\Omega = \emptyset$ then $V_\lambda = 0$. If $\Omega \neq \emptyset$ then $V_\lambda = \vert \bd^\star \Omega\vert \neq 0$, and moreover, $V_\lambda$ is induced by a smooth surface with constant mean curvature $\lambda$ whose mean curvature vector points into $\Omega$. 
\end{itemize}
The minimal surfaces may have tangential intersection with the CMC surface.  Likewise, the CMC surface may have tangential intersections with itself but it never crosses itself. 
\end{theorem}

We can now complete the proof of Theorem \ref{main}. 

\begin{theo}
\label{main}
Let $(M^{n+1},g)$ be a closed Riemannian manifold with $3\le n+1\le 7$. Fix a number $p\in \N$. There are
\begin{itemize}
\item[(i)] a Caccioppoli set $\Omega\subset M$ with $\vol(\Omega) = \frac 1 2 \vol(M)$ whose boundary is smooth and almost-embedded with constant mean curvature,
\item[(ii)]  a (possibly empty) collection of smooth, disjoint minimal surfaces $\Sigma_1,\hdots,\Sigma_k \subset M \setminus \Omega$, 
\item[(iii)] and positive integers $\theta_0 \in \Z$ and $\theta_1,\hdots,\theta_k\in 2\Z$ 
\end{itemize}
such that
$
\tilde c(p) = \theta_0\area(\bd \Omega) + \theta_1\area(\Sigma_1) + \hdots + \theta_k\area(\Sigma_k).
$
Moreover, $\theta_0 = 1$ unless $\bd \Omega$ is also a minimal surface. 
\end{theo}

\begin{proof} Choose a sequence $\eps_k \to 0$ so that $\tilde c(\eps_k,p)\to \tilde c(p)$.  Let $u_{\eps_k}$ be the critical points of $F_{\eps_k,\lambda_{\eps_k}}$ constructed in Proposition \ref{ace}.  Then $E_{\eps_k}(u_{\eps_k})$ and $\ind_{F_{\eps_k,\lambda_{\eps_k}}}(u_{\eps_k})$ are uniformly bounded.  According to Hutchinson and Tonegawa \cite{HT} Section 6.1 and Lemma 3.4 in \cite{XC}, the Lagrange multipliers $\lambda_{\eps_k}$ are also uniformly bounded and there is a constant $K > 0$ such that $\|u_{\eps_k}\|_{L^\infty}\le K$ for all $k$.  For the interested reader, we include the details of the argument in the appendix. Applying Theorem \ref{bw} to the sequence $u_{\eps_k}$ now yields the result. 
\end{proof}

\section{Appendix} 

The goal of the appendix is to prove the following proposition. We largely follow the sketch in \cite{HT} section 6.1, giving details as appropriate.  The proof that the Lagrange multipliers are bounded depends on work of Chen \cite{XC}. 

\begin{prop}
Assume that the potential $W$ satisfies the growth condition (vii). Choose a sequence $\eps_k\to 0$ and let $u_{\eps_k}$ be a critical point of $F_{\eps_k,\lambda_{\eps_k}}$.  Assume that $E_{\eps_k}(u_{\eps_k})$ is uniformly bounded.  Then the Lagrange multipliers $\lambda_{\eps_k}$ are uniformly bounded and $\|u_{\eps_k}\|_{L^\infty}$ is also uniformly bounded.
\end{prop} 

\begin{proof} 
The first step is to check that each $u_{\eps_k}$ is smooth.  Recall that $\vert W'(u_{\eps_k})\vert \le C\vert u_{\eps_k}\vert^{q-1}$ for $\vert u_{\eps_k}\vert \ge \beta$.  For simplicity, we will give the argument assuming $3 \le n+1 \le 5$.    In this case, by the Sobolev embedding theorem, $u_{\eps_k}$ belongs to $L^{10/3}$.  
It follows that $W'(u_{\eps_k}) \in L^{10/(3q-3)}$. 
Hence by elliptic regularity we have $u_{\eps_k} \in W^{2,q_1}$ for $q_1 = 10/(3q-3)$.  Note that 
\[
\frac{n+1}{q_1} \le \frac{3}{2}(q-1) \le \frac{18}{10} < 2. 
\]
Thus by the Sobolev inequalities we obtain that $u_{\eps_k}$ is H\"older continuous. Standard elliptic regularity then implies that $u_{\eps_k}$ is smooth.  The cases $6 \le n+1\le 7$ are handled similarly.  One applies the Sobolev inequalities together with elliptic regularity several times.  Each application improves the regularity of $u_{\eps_k}$ until eventually one obtains $u_{\eps_k} \in W^{2,q_1}$ with $2 > (n+1)/q$. This gives H\"older continuity of $u_{\eps_k}$, and standard elliptic regularity then implies that $u_{\eps_k}$ is smooth. Note at this point we do not have uniform $L^\infty$ estimates on $u_{\eps_k}$.

To prove that the Lagrange multipliers are bounded we follow \cite{XC} Lemma 3.4.  Note that \cite{XC} Lemma 3.4 is proved for domains in Euclidean space.  Some addition difficulties arise in adapting the mollifier arguments used in \cite{XC} to the case of a closed manifold.  We need to prove a sequence of lemmas.  In what follows, $C$ denotes a positive constant that is allowed to change from line to line. 

\begin{lem}
\label{l2c}
We have 
\[
\int_M (\vert u_{\eps_k}\vert - 1)^2 \to 0, \quad \text{as } k\to\infty.
\]
\end{lem}

\begin{proof}
Define the sets 
\begin{gather*}
A_1 =\{\vert \vert u_{\eps_k}\vert - 1\vert \le \eps_k^{1/4}\},\\
A_2 =  \{\eps_k^{1/4} \le \vert\vert u_{\eps_k} \vert - 1\vert \le \beta - 1\},\\  
A_3 = \{\vert u_{\eps_k}\vert \ge \beta\}.
\end{gather*} 
We will estimate the integral over the sets $A_1$, $A_2$, and $A_3$ separately.  For $A_1$, we have 
\[
\int_{A_1} (\vert u_{\eps_k} \vert - 1)^2 \le \eps_k^{1/2} \vol(M). 
\]
Regarding $A_2$, property (iii) of $W$ imply that there is a constant $c > 0$ independent of $k$ such that $\vert W(x) \vert  \ge c \eps_k^{1/2}$ whenever $\vert \vert x\vert - 1\vert \ge \eps_k^{1/4}$.   
Therefore the set $A_2$ has measure at most 
$
{ \eps_k^{1/2} }{c\inv} E_{\eps_k}(u_{\eps_k})
$
and so 
\[
\int_{A_2} (\vert u_{\eps_k} \vert - 1)^2 \le (\beta-1)^2 c\inv \eps_k^{1/2} E_{\eps_k}(u_{\eps_k}). 
\]
It remains to estimate the integral over $A_3$.  Without loss of generality we can assume that $\beta \ge 5$.   Remember that $C \vert W(x)\vert \ge  \vert x\vert^q$ for $\vert x\vert \ge \beta$ and so 
\[
(\vert u_{\eps_k}\vert - 1)^2 \le \vert u_{\eps_k}\vert^2 \le \vert u_{\eps_k}\vert^q \le C \vert W(u_{\eps_k})\vert
\]
whenever $\vert u_{\eps_k}\vert \ge \beta$.    Thus we have the estimate 
\[
\int_{A_3} (\vert u_{\eps_k} \vert - 1)^2 \le \int_{A_3} C \vert W(u_{\eps_k})\vert \le C \eps_k E_{\eps_k}(u_{\eps_k}). 
\]
Combining these three estimates shows that  
\[
\int_M (\vert u_{\eps_k}\vert - 1)^2 \le  C\eps_k^{1/2}(1+E_{\eps_k}(u_{\eps_k})).
\]
The right hand side goes to 0 as $k\to \infty$. 
\end{proof} 

Define 
\[
\Phi(s) = \int_0^s \sqrt{W(s)/2}\, ds.
\]
Let $w_{\eps_k} = \Phi\circ u_{\eps_k}$.  It is easy to check that $\grad w_{\eps_k}$ is uniformly bounded in $L^1$ (see \cite{HT}). 

\begin{lem} 
\label{l1} 
There is a constant $C > 0$ such that $\vert x- y\vert^2 \le C \vert \Phi(x)-\Phi(y)\vert$ for all $x,y\in \R$.  
\end{lem} 

\begin{proof} 
We check a number of cases.  First suppose that $x, y \ge \beta$ and assume without loss of generality that $x \ge y$.  Then 
\begin{align*}
\vert \Phi(x)-\Phi(y)\vert &= \int_y^x \sqrt{W(s)/2} \, ds \ge \frac{1}{\sqrt 2} \int_y^x  s^{q/2}\, ds \\ 
&\ge \frac{1}{\sqrt 2} \int_y^x  s \, ds = \frac{1}{2\sqrt 2}(x^2-y^2) \ge \frac{1}{2\sqrt 2} \vert x-y\vert^2,
\end{align*}
where the last inequality uses the fact that $x \ge y \ge 0$ and so $x+y \ge x-y$. The same argument works in the case where $x,y\le -\beta$.  

Now suppose that $x\ge \beta$ and $y \le -\beta$. Let $a = \int_{-\beta}^\beta \sqrt{W(s)/2}\, ds > 0$. Then 
\begin{align*}
\vert \Phi(x) - \Phi(y)\vert &= \int_{\beta}^x \sqrt{W(s)/2}\, ds + a +  \int_{y}^{-\beta} \sqrt{W(s)/2}\, ds\\
&\ge \frac{1}{\sqrt 2} \int_{\beta}^x s\, ds  + a + \frac{1}{\sqrt 2} \int_{y}^{-\beta} (-s)\, ds\\
&\ge \frac{1}{2\sqrt 2}(x^2-\beta^2) + a + \frac{1}{2\sqrt 2}(y^2-\beta^2). 
\end{align*}
We have 
\begin{align*}
\vert x-y\vert^2 \le 2(x^2+y^2) \le 4\beta^2 + 2(x^2 -\beta^2) + 2(y^2-\beta^2) \le C\vert \Phi(x) - \Phi(y)\vert.
\end{align*}
The same argument works if $x \le -\beta$ and $y \ge \beta$. 

It remains to handle the case when $-\beta \le x,y\le \beta$. Assume for contradiction there are two sequences $x_j,y_j\in [-\beta,\beta]$ such that 
\begin{equation}
\label{eq3} 
\vert x_j-y_j\vert^2 > j \vert \Phi(x_j) - \Phi(y_j)\vert. 
\end{equation}
Passing to a subsequence if necessary, we can assume that $x_j \to x\in [-\beta,\beta]$ and $y_j \to y\in [-\beta,\beta]$. It is clear from (\ref{eq3}) that $x= y$. Passing to the limit in 
\[
\vert x_j - y_j\vert \ge j\frac{\vert \Phi(x_j) - \Phi(y_j)\vert}{\vert x_j-y_j\vert}
\]
we obtain that $\Phi'(x) = 0$. Thus $W(x) = 0$ and so $x = \pm 1$.  Without loss of generality assume that $x = 1$.  Note that there is a constant $C > 0$ such that $W(s) \ge C \vert s - 1\vert^2$ for $s$ close to 1.  Thus 
\begin{align*}
\vert \Phi(x_j) - \Phi(y_j)\vert = \left\vert \int_{y_j}^{x_j} \sqrt{W(s)/2} \, ds\right\vert \ge C \left\vert \int_{y_j}^{x_j} \vert s-1\vert \, ds\right\vert.
\end{align*}
If $x_j \ge y_j \ge 1$ then 
\begin{align*}
\left\vert \int_{y_j}^{x_j} \vert s-1\vert \, ds\right\vert &= \int_{y_j}^{x_j} s-1\, ds = \frac{1}{2}\left[(x_j-1)^2 - (y_j-1)^2 \right]\\
&= \frac 1 2 \left[ (x_j+y_j-2)(x_j-y_j)\right] \ge \frac{1}{2}(x_j-y_j)^2,
\end{align*}
which combined with the previous equation yields a contradiction.  If $x_j \ge 1 \ge y_j$ then 
\begin{align*}
\left\vert \int_{y_j}^{x_j} \vert s-1\vert \, ds\right\vert &= \int_{y_j}^1 1-s\, ds + \int_1^{x_j} s-1\, ds \\
&= \frac 1 2 \left[(1-y_j)^2 + (x_j-1)^2\right] \ge \frac{1}{4}(x_j-y_j)^2
\end{align*}
since either $x_j - 1 \ge \frac{1}{2}(x_j - y_j)$ or $1 - y_j \ge \frac{1}{2}(x_j-y_j)$. Again this gives a contradiction. The remaining possibilities likewise lead to contradiction and the lemma is proved. 
\end{proof}

For each $\eta \in (0,1)$, let $u_{\eps_k,\eta}$ be a mollified version of $u_{\eps_k}$.  More precisely, choose an isometric embedding of $M$ into $\R^m$.  Let $N$ be a small tubular neighborhood of $M$ where the nearest point projection $\Pi\f N\to M$ is well-defined and a submersion.  Let $B_r(x)$ denote the open ball of radius $r$ centered at $x\in \R^m$.  Also let $B_1$ denote the open unit ball centered at the origin in $\R^m$.  Choose a non-negative smooth function $\rho\f B_1 \to \R$ which is compactly supported in $B_1$ and satisfies 
\[
\int_{B_1} \rho(y)\, d\mathcal L^{m}(y) = 1. 
\]
Extend $u_{\eps_k}$ to a function $v_{\eps_k}$ on $N$ by setting $v_{\eps_k} = u_{\eps_k}\circ \Pi$. Then let  
\[
u_{\eps_k,\eta}(x) =  \int_{B_1} \rho(y) v_{\eps_k}(x-\eta y)   \, d\mathcal L^{m}(y), \quad x\in M
\]
be the mollified version of $u_{\eps_k}$. 

\begin{lem} 
\label{linf}
For each fixed $\eta$, there is a uniform bound $\|u_{\eps_k,\eta}\|_{L^\infty(M)} \le C(\eta)$.
\end{lem}

\begin{proof} 
Observe that 
\begin{align*}
\vert u_{\eps_k,\eta}(x) \vert &\le \int_{B_1} \rho(y) \vert v_{\eps_k}(x-\eta y)\vert \, d\mathcal L^{m}(y)\\ 
&\le 1 + \int_{B_1} \rho(y) \big\vert \vert v_{\eps_k}(x-\eta y)\vert - 1\big\vert\, d\mathcal L^{m}(y)\\
&\le 1 + C\eta^{-(n+1)} \int_{B_\eta(x)} \big\vert \vert v_{\eps_k}(z)\vert - 1\big\vert\, d\mathcal L^{m}(z).
\end{align*}
By the co-area formula we have
\begin{align*}
&\int_{B_\eta(x)} \big\vert \vert v_{\eps_k}(z)\vert - 1\big\vert\, d\mathcal L^{m}(z) \\
&\qquad \le  C \int_{B_\eta(x)} \big\vert \vert v_{\eps_k}(z)\vert - 1\big\vert\, J\Pi(z) \, d\mathcal L^{m}(z)\\
&\qquad \le C \int_{\Pi(B_\eta(x))} \int_{p\in \Pi\inv(q)} \big\vert \vert v_{\eps_k}(p)\vert - 1 \big\vert \, d\mathcal H^{m-n-1}(p) \, d\mathcal H^{n+1}(q)\\
&\qquad \le C \int_{\Pi(B_\eta(x))} \big\vert \vert u_{\eps_k}(q)\vert - 1 \big\vert\, d\mathcal H^{n+1}(q).
\end{align*}
Inserting this into the previous equation and using Lemma \ref{l2c} gives the result.
\end{proof}

\begin{lem} 
\label{linfd} 
For each fixed $\eta$, there is a uniform bound $\|\grad u_{\eps_k,\eta}\|_{L^\infty(M)} \le C(\eta)$.
\end{lem}

\begin{proof} Fix an index $i\in\{1,\hdots,m\}$. Note the formula for $u_{\eps_k,\eta}(x)$ also makes sense for $x\in N$ so we can regard $u_{\eps_k,\eta}$ as a function defined on $N$.  For $x\in M$ we have
\begin{align*}
\bd_i u_{\eps_k,\eta}(x) &= \int_{B_1} \bd_i \rho(y) v_{\eps_k}(x-\eta y)\, d\mathcal L^m(y).
\end{align*}
Thus we have 
\begin{align*}
\vert \bd_i u_{\eps_k,\eta}(x)\vert &\le \int_{B_1} \vert \bd_i\rho(y)\vert \vert v_{\eps_k}(x-\eta y)\vert \, d\mathcal L^m(y)\\ 
&\le C + \int_{B_1} \vert \bd_i\rho(y)\vert (\vert v_{\eps_k}(x-\eta y)\vert - 1)\, d\mathcal L^m(y) \\
&\le C + C\eta^{-(n+1)} \int_{B_{\eta(x)}} \big\vert \vert v_{\eps_k}(z)\vert - 1 \big\vert\, d\mathcal L^m(z).
\end{align*}
Using the coarea formula as in the proof of the previous lemma now gives the result.
\end{proof}  

\begin{lem}
\label{mc}
There is a uniform bound $\|u_{\eps_k,\eta} - u_{\eps_k}\|^2_{L^2(M)} \le C \eta$ for $k \ge K(\eta)$. 
\end{lem}

\begin{proof}
Observe that 
\begin{align*}
\int_M \vert u_{\eps_k,\eta} - u_{\eps_k}\vert^2 &= \int_M \left\vert \int_{B_1} \rho(y)v_{\eps_k}(x-\eta y)\, d\mathcal L^{m}(y) - u_{\eps_k}(x)\right\vert^2 \,d\mathcal H^{n+1}(x)\\
&\le \int_M \int_{B_1} \rho(y) \vert v_{\eps_k}(x-\eta y) - u_{\eps_k}(x)\vert^2 \, d\mathcal L^{m}(y)\, d\mathcal H^{n+1}(x)\\
&\le C \int_M \int_{B_1} \rho(y) \vert f(x-\eta y) - f(x)\vert \, d\mathcal L^{m}(y)\, d\mathcal H^{n+1}(x),
\end{align*} 
where $f = \Phi\circ u_{\eps_k}\circ \Pi = w_{\eps_k}\circ \Pi$ and we used Lemma \ref{l1} to get the last inequality. By Fubini's theorem we get 
\begin{align*}
&\int_M \int_{B_1} \rho(y) \vert f(x-\eta y) - f(x)\vert \, d\mathcal L^{m}(y)\, d\mathcal H^{n+1}(x) \\
&\qquad \le \eta \int_M \int_{B_1} \int_0^1 \rho(y) \vert \grad f(x- t \eta y)\vert \, \mathcal \,dt \, L^{m}(y)\, d\mathcal H^{n+1}(x)\\
&\qquad = \eta \int_{B_1} \rho(y) \int_0^1 \int_M \vert \grad f(x-t\eta y)\vert \, d\mathcal H^{n+1}(x) \, dt\, d\mathcal L^{m}(y).
\end{align*}
Now for fixed $y$ and $t$, we have 
\[
\int_M \vert \grad f(x-t\eta y)\vert \, d\mathcal H^{n+1}(x) = \int_{M-t\eta y} \vert \grad f(z)\vert \, d\mathcal H^{n+1}(z).
\]
Provided $\eta$ is small enough, the map $\Pi\f M-t\eta y \to M$ is a diffeomorphism and so by the change of variables forumla 
\begin{align*}
\int_{M-t\eta y} \vert \grad f(z)\vert \, d\mathcal H^{n+1}(z) &\le \int_{M-t\eta y} \vert \grad w_{\eps_k}(\Pi(z))\vert  \, d\mathcal H^{n+1}(z) \\
&\le C\int_{M-t\eta y} \vert \grad w_{\eps_k}(\Pi(z))\vert J\Pi(z) \, d\mathcal H^{n+1}(z) \\
&= C\int_M \vert \grad w_{\eps_k}(q)\vert \, d\mathcal H^{n+1}(q).
\end{align*} 
Thus we obtain 
\begin{align*}
&\eta \int_{B_1} \rho(y) \int_0^1 \int_M \vert \grad f(x-t\eta y)\vert \, d\mathcal H^{n+1}(x) \, dt\, d\mathcal L^{m}(y)\\
&\qquad \le C\eta \int_{B_1} \rho(y) \|w_{\eps_k}\|_{L^1(M)} \, d\mathcal L^{m}(y) \le C\eta. 
\end{align*}
Putting everything together we get 
\[
\int_M \vert u_{\eps_k,\eta} - u_{\eps_k}\vert^2  \le C\eta,
\]
and the result follows. 
\end{proof} 

\begin{lem}
\label{avg} 
Let $\overline u_{\eps_k,\eta}$ be the average of $u_{\eps_k,\eta}$. 
There is a uniform bound $\vert \overline u_{\eps_k,\eta}\vert\le C\eta^{\frac{1}{2}}$ for $k\ge K(\eta)$. 
\end{lem} 

\begin{proof}
Observe that 
\begin{align*}
\overline u_{\eps_k,\eta} &= \frac{1}{\vol(M)} \int_M u_{{\eps_k},\eta}\, = \frac{1}{\vol(M)} \int_M (u_{\eps_k,\eta}- u_{\eps_k}).
\end{align*}
Now the result follows from Lemma \ref{mc} and H\"older's inequality.
\end{proof}

Now fix an $\eta$ to be specified later. Given a large integer $k$, let $\psi$ be the solution to  
\[
-\lap \psi = u_{\eps_k,\eta} - \overline u_{\eps_k,\eta}, \quad \int_{M} \psi = 0.
\]
By Lemmas \ref{linf}, \ref{linfd}, and \ref{avg} the right hand side of the above PDE is uniformly bounded in $C^{1}$.  Therefore by elliptic regularity, $\psi$ is bounded in $C^{2}$ by a constant that depends on $\eta$ but not on $k$.  Note that $u_{\eps_k}$ satisfies the PDE 
\[
- {\eps_k} \lap u_{\eps_k} + \frac{W'(u_{\eps_k})}{\eps_k} = \lambda_{\eps_k}.
\]
Multiplying by $\grad \psi \cdot \grad u_{\eps_k}$ and integrating yields 
\begin{equation}
\label{eq1}
\lambda_{\eps_k} \int_M \grad \psi \cdot \grad u_{\eps_k} =  \int_M  \grad \psi \cdot \grad u_{\eps_k}\left(-{\eps_k}\lap u_{\eps_k} + \frac{W'(u_{\eps_k})}{\eps_k}\right).
\end{equation}
Observe that 
\[
\int_M \grad \psi\cdot \grad u_{\eps_k} \frac{W'(u_{\eps_k})}{\eps_k} = \int_M \grad \psi\cdot \grad \left(\frac{W(u_{\eps_k})}{\eps_k}\right) = - \int_M \frac{W(u_{\eps_k})}{\eps_k}\lap \psi.
\]
Also, by the integration by parts formula for the Hessian, we have 
\begin{align*}
\eps_k \int_M D^2 \psi(\grad u_{\eps_k},\grad u_{\eps_k}) = -\eps_k \int_M \grad \psi \cdot \grad u_{\eps_k} \lap u_{\eps_k} + \frac{\eps_k}{2}\int_M \vert \grad u_{\eps_k}\vert^2 \lap \psi. 
\end{align*}
Thus we have the following formula for the right hand side of (\ref{eq1}):
\begin{align*}
&\int_M  \grad \psi \cdot \grad u_{\eps_k}\left(-{\eps_k}\lap u_{\eps_k} + \frac{W'(u_{\eps_k})}{\eps_k}\right) \\
&\qquad = \eps_k \int_M D^2\psi(\grad u_{\eps_k},\grad u_{\eps_k}) - \int_M \left(\frac{\eps_k}{2}\vert \grad u_{\eps_k}\vert^2 + \frac{W(u_{\eps_k})}{\eps_k}\right)\lap \psi.
\end{align*}
Since $\|\psi\|_{C^2} \le C(\eta)$, this gives a bound 
\[
\left\vert \int_M  \grad \psi \cdot \grad u_{\eps_k}\left(-{\eps_k}\lap u_{\eps_k} + \frac{W'(u_{\eps_k})}{\eps_k}\right)\right\vert \le C(\eta) E_{\eps_k}(u_{\eps_k}).
\]
We now turn attention to the left hand side of (\ref{eq1}).  Integrating by parts, we have 
\[
\lambda_{\eps_k}\int_M \grad \psi \cdot \grad u_{\eps_k} = -\lambda_{\eps_k} \int_M u_{\eps_k}\lap \psi = \lambda_{\eps_k} \int_M u_{\eps_k}(u_{\eps_k,\eta} - \overline u_{\eps_k,\eta}).
\]
Now observe that 
\begin{align*}
\int_M u_{\eps_k}(u_{\eps_k,\eta} - \overline u_{\eps_k,\eta}) = \int_M u_{\eps_k}(u_{\eps_k,\eta} - u_{\eps_k}) + \int_M (u_{\eps_k}^2-1) - \overline u_{\eps_k,\eta} \int_M u_{\eps_k} + \vol(M).
\end{align*}
By Lemmas \ref{l2c} and \ref{avg} and H\"older's inequality, we can select $\eta$ small enough that 
\[
\left\vert\int_M u_{\eps_k}(u_{\eps_k,\eta} - u_{\eps_k})\right\vert \le \frac{1}{4}\vol(M), \quad 
\left\vert \overline u_{\eps_k,\eta} \int_M u_{\eps_k}\right\vert \le \frac{1}{4}\vol(M). 
\]
By Lemma \ref{l2c} and H\"older's inequality, we have
\[
\left\vert \int_M (u_{\eps_k}^2 - 1)\right\vert \le \int_M \vert \vert u_{\eps_k}\vert - 1\vert (\vert u_{\eps_k}\vert + 1) \le  \frac{1}{4}\vol(M)
\]
for $k$ large enough. It follows that 
\[
\int_M \grad \psi \cdot \grad u_{\eps_k} \ge \frac{1}{4}\vol(M)
\]
for large enough $k$. 
Using equation (\ref{eq1}) then gives an upper bound on $\lambda_{\eps_k}$ which is independent of $k$. 

To complete the proof of the proposition, it remains to show that $\|u_{\eps_k}\|_{L^\infty}$ is uniformly bounded.  Let $M_{\eps_k} = M/\eps_k$.  Define the rescaled functions $f_{\eps_k}\f M_{\eps_k} \to \R$ by $f_{\eps_k}(x) = u_{\eps_k}(\eps_k x)$.  Then $f_{\eps_k}$ solves 
\[
-\lap f_{\eps_k} + W'(f_{\eps_k}) = \eps_k \lambda_{\eps_k}.
\]
Fix some $p \ge 1$.  Multiplying the equation by $\vert f_{\eps_k}\vert^{p-1} f_{\eps_k}$ and integrating gives 
\begin{equation}
\label{eq4}
\int_{M_{\eps_k}} p\vert f_{\eps_k}\vert^{p-1}\vert \grad f_{\eps_k}\vert^2 + \int_{M_{\eps_k}} W'(f_{\eps_k})\vert f_{\eps_k}\vert^{p-1}f_{\eps_k} = \eps_k\lambda_{\eps_k} \int_{M_{\eps_k}} \vert f_{\eps_k}\vert^{p-1} f_{\eps_k}.
\end{equation}
Now for $\vert x\vert \ge \beta$ we have $W'(x) \vert x\vert^{p-1} x \ge C\vert x\vert^{p+q-1}$.  Since $\eps_k \lambda_{\eps_k} \to 0$, it follows from (\ref{eq4}) that  
\[
\int_{\{f_{\eps_k} \ge \beta\}} \vert f_{\eps_k}\vert^{p+q-1} \le \frac{1}{2} \int_{M_{\eps_k}} \vert f_{\eps_k}\vert^p
\]
assuming $k$ is large enough.  This implies that 
\[
\int_{M_{\eps_k}} \vert f_{\eps_k}\vert^{p+q-1} \le C \beta^{p+q-1} + \frac{1}{2} \int_{M_{\eps_k}} \vert f_{\eps_k}\vert^p.
\]
By induction, for any positive integer $r$, this gives 
\begin{align*}
\int_{M_{\eps_k}} \vert f_{\eps_k}\vert^{2+r(q-1)} &\le \frac{C \beta^2}{2^r} \sum_{j=1}^r (2\beta^{q-1})^j + 2^{-r}\int_{M_{\eps_k}} \vert f_{\eps_k}\vert^2\\
&= \frac{C\beta^2}{2^r}\left(\frac{2\beta^{q-1}(2^r \beta^{r(q-1)}-1)}{2\beta^{q-1}-1}\right) + 2^{-r}\int_{M_{\eps_k}} \vert f_{\eps_k}\vert^2\\
&\le C\beta^{2+r(q-1)} + 2^{-r}\int_{M_{\eps_k}} \vert f_{\eps_k}\vert^2. 
\end{align*}
Raising both sides to the power $(2+r(q-1))\inv$ and then sending $r\to \infty$ gives the bound $\|f_{\eps_k}\|_{L^\infty(M_{\eps_k})} \le \beta$ provided $k$ is large enough.  This implies that $\|u_{\eps_k}\|_{L^\infty(M)} \le \beta$ for all large $k$, as needed. 
\end{proof}

\bibliographystyle{plain}
\bibliography{cmc-doubling}

\end{document}